\documentclass{article}
\usepackage[margin=1in]{geometry}
\usepackage{authblk}
\usepackage[T1]{fontenc}
\usepackage{amsmath,amssymb,mathtools}
\usepackage{amsthm}
\usepackage{thmtools}
\usepackage{listings}
\usepackage{textcomp}
\makeatletter
\@ifundefined{newcounteralias}{}{%
  \renewcommand\thmt@autorefsetup{%
    \@xa\def\csname\thmt@envname autorefname\@xa\endcsname
    \@xa{\thmt@thmname}}%
}
\makeatother
\usepackage{booktabs}
\usepackage{microtype}
\usepackage[round]{natbib}
\usepackage{algorithm}
\usepackage{algpseudocode}
\usepackage[hidelinks,hypertexnames=false]{hyperref}
\usepackage[capitalize,noabbrev,nameinlink]{cleveref}
\usepackage{autonum}
\usepackage{todonotes}
\usepackage{mathtools}
\usepackage{tikz}
\usetikzlibrary{arrows.meta}
\usepackage{etoolbox}

\newtheorem{theorem}{Theorem}[section]
\newtheorem{proposition}[theorem]{Proposition}
\newtheorem{lemma}[theorem]{Lemma}

\newtheorem{definition}[theorem]{Definition}
\theoremstyle{definition}
\newtheorem{assumption}[theorem]{Assumption}
\newtheorem{example}[theorem]{Example}
\AtEndEnvironment{example}{\unskip\nobreak\hfill$\diamond$}
\newtheorem{remark}[theorem]{Remark}
\AtEndEnvironment{remark}{\unskip\nobreak\hfill$\diamond$}

\crefname{theorem}{Theorem}{Theorems}
\Crefname{theorem}{Theorem}{Theorems}
\crefname{proposition}{Proposition}{Propositions}
\Crefname{proposition}{Proposition}{Propositions}
\crefname{definition}{Definition}{Definitions}
\Crefname{Definition}{Definition}{Definitions}
\crefname{remark}{Remark}{Remarks}
\Crefname{remark}{Remark}{Remarks}
\crefname{lemma}{Lemma}{Lemmas}
\Crefname{lemma}{Lemma}{Lemmas}
\crefname{corollary}{Corollary}{Corollaries}
\Crefname{corollary}{Corollary}{Corollaries}
\crefname{assumption}{Assumption}{Assumptions}
\Crefname{assumption}{Assumption}{Assumptions}
\crefname{example}{Example}{Examples}
\Crefname{example}{Example}{Examples}
\crefname{appendix}{Appendix}{Appendices}
\Crefname{appendix}{Appendix}{Appendices}
\crefname{algorithm}{Algorithm}{Algorithms}
\Crefname{algorithm}{Algorithm}{Algorithms}

\newif\ifshowrevisions
\showrevisionsfalse
\newcommand{\revisioncolor}{\ifshowrevisions\color{red}\fi}

\newenvironment{revision}{\par\begingroup\revisioncolor}{\par\endgroup}

\DeclarePairedDelimiter{\prn}{(}{)}
\DeclarePairedDelimiter{\set}{\{}{\}}
\DeclarePairedDelimiterX{\Set}[2]{\{}{\}}{\,#1\mid#2\,}
\DeclarePairedDelimiter{\abs}{\lvert}{\rvert}
\DeclarePairedDelimiter{\norm}{\lVert}{\rVert}
\DeclarePairedDelimiter{\inpr}{\langle}{\rangle}
\DeclarePairedDelimiter{\brc}{[}{]}

\newcommand{\infconv}{\,\square\,}

\newcommand{\mdp}{\mathcal{M}}
\newcommand{\distf}{\mathcal{D}}
\newcommand{\mc}[1]{\mathcal{M}[#1]}
\newcommand{\rp}[2]{\mathrm{RPr}^{#1}_{#2}}
\newcommand{\pathmc}[2]{\mathrm{Path}^{#1}_{#2}}
\newcommand{\wrp}[1]{\mathrm{WRPr}^{\max}_{#1}}

\newcommand{\R}{\mathbb{R}}

\newcommand{\dws}{\mathrm{dw}}

\DeclareMathOperator{\dist}{dist}

\DeclareMathOperator{\proj}{proj}

\newcommand{\pos}[1]{\brc*{#1}_+}

\newcommand{\SAFE}{\textnormal{\textsc{safe}}}
\newcommand{\UNKNOWN}{\textnormal{\textsc{unknown}}}
\newcommand{\UNSAFE}{\textnormal{\textsc{unsafe}}}

\title{Convex-Analytic View of Query Selection\\for Safety Verification of Markov Decision Processes\thanks{Authors are listed in alphabetical order.}}
\author[1,2,3]{Han Bao}
\author[4,5,3]{Shinsaku Sakaue}
\author[5]{Kazuki Watanabe}
\affil[1]{The Institute of Statistical Mathematics}
\affil[2]{Tohoku University}
\affil[3]{RIKEN AIP}
\affil[4]{CyberAgent}
\affil[5]{National Institute of Informatics}
\date{\today}

\begin{document}

\maketitle

\begin{abstract}
Safety verification of a Markov decision process (MDP) asks whether the probabilities of reaching several failure states jointly lie in a convex safe region under every scheduler.
\citet{watanabe2026posterior} prove safety by refining an outer approximation of the failure probabilities, where each query bounds a nonnegative weighted sum of the probabilities.
Their rule selects weights by the largest Euclidean violation.
We present a unified framework to measure violation in an explicitly chosen error unit: directions are scored by the support-function gap between the outer approximation and the safe region, normalized by the support function of the unit.
We show that the maximal score, when positive, equals the smallest enlargement of the safe region in multiples of that unit that contains the outer approximation.
The Euclidean unit recovers the earlier rule; two new rules are generated by the unit cube and the safe region itself.
We experimentally evaluate how the choice of measurement rule affects the performance of safety verification for MDPs.
\end{abstract}

\section{Introduction}
\label{sec:intro}

Consider a robot that carries out a mission consisting of several tasks. At each step it chooses a route or an operating mode, the environment changes at random, and each task may end in a catastrophe of one of $d$ kinds, such as a collision, a fall, or a loss of communication; the mission ends when all tasks are completed or when a catastrophe occurs. A safety requirement constrains the probabilities $p_1,\dots,p_d$ of the $d$ catastrophes jointly: each $p_i$ may be at most a threshold $\theta_i$, or a severity-weighted total $\sum_iw_ip_i$ may be at most a threshold $\theta$. The permissible vectors $\prn{p_1,\dots,p_d}$ then form a safe region, represented by a nonempty compact convex set. The requirement has to hold not only for one chosen control policy but for every way of resolving the choices, because the choices may be taken by a component whose behavior is unknown or adversarial.

The robot and environment form a Markov decision process, and a scheduler resolves the choices; then we aim to ensure that the vector of failure probabilities of every scheduler lies in the safe region.
The exact set of failure probabilities is unavailable a priori due to infinitely many schedulers and the huge state space.
Here, \emph{counterexample-guided abstraction refinement (CEGAR)}~\citep{BallR02} is useful, which iteratively synthesizes the largest \emph{weighted} probability.\looseness=-1

Specifically, for a chosen nonnegative weight vector~$w$, the largest weighted probability $\sup_\sigma\sum_iw_ip^\sigma_i$ over all schedulers $\sigma$ is a single-objective problem: for a finite process it is an optimal reachability problem that standard algorithms can solve \citep{baier2008principles,etessami2008multi}, and for infinite-state probabilistic programs invariant synthesis can certify its upper bound~(see, e.g.,~\cite{ChakarovS13,Chatterjee0020,TakisakaOUH21}).
Each query $w$ yields a certified bound $u\ge\sup_\sigma\sum_{i} w_i p_i^\sigma$.
Every attainable vector is in the corresponding half-space
$\Set{p}{\sum_{i} w_i p_i\le u}$.
We form an \emph{outer approximation} by intersecting an initial set
containing all attainable vectors with the returned half-spaces.
A new bound removes points precisely when the current approximation
contains $p$ with $\sum_{i} w_i p_i>u$.
Once the approximation lies inside the safe region,
the requirement is proved for every scheduler.

In the above CEGAR approach, each query $w$ requires a full analysis of the process and therefore is often expensive, and a query in a direction in which the approximation is already tight brings nothing.
We therefore seek query directions that reveal violations of
the safe region by the current outer approximation.

Using convex analysis, we develop a framework for quantifying
these violations and choosing the next query direction.
By reincarnating the idea of \citet{Hormander1955}, we measure the set difference by the \emph{difference of the support functions} of two convex sets:
writing the support functions of two nonempty compact convex sets $U$ and $S$ by $h_U$ and $h_S$ respectively, the basic course of convex analysis teaches us the equivalence $U\subseteq S \iff h_U\le h_S$~\citep{rockafellar1970convex}.
Imagine $U$ and $S$ are the outer approximation and safe region, respectively.
Among nonnegative weight vectors whose entries sum to one,
a natural query is a maximizer of $h_U(w)-h_S(w)$.
To compare violations across directions in a chosen geometry, we normalize the support gap to consider $(h_U-h_S)/h_E$, where $E$ is a user-chosen convex set, called an \emph{error unit}.
If $E$ is the nonnegative part of the Euclidean unit ball, the normalizer $h_E(w)$ equals $\norm{w}_2$, and we recover the projection rule of \citet{watanabe2026posterior}.
If $E$ is the unit cube $\brc{0,1}^d$, the normalizer $h_E(w)$ equals $\norm{w}_1$, which is one whenever the entries of $w$ sum to one.
If $E=S$, the normalizer is the support function $h_S$ of the safe region itself, and the largest value of the normalized support gap is expressed through the \emph{gauge function} of $S$.
We show that the query based on $E=[0,1]^d$ and $E=S$ can typically be computed more efficiently, which actually accelerates safety verification empirically.

The paper is organized as follows.
\cref{sec:mdp-setup} gives a self-contained introduction to probabilistic model checking, for those who are not familiar with this topic.
In \cref{sec:setup}, we give our general problem statement of the set containment test, which can be read without the context of probabilistic model checking.
\cref{sec:scores} presents our query framework based on the support functions, and instantiates it with various choices of the error unit.
In \cref{sec:experiments}, we evaluate different query strategies on a common oracle that relies on the existing probabilistic model checker \emph{Storm}~\citep{HenselJKQV22}.

\paragraph{Notation.}
We write $\pos{\cdot} = \max\set{\cdot, 0}$.
For $d\ge1$, we write $\triangle_d\coloneqq\Set{p\in[0, 1]^d}{\sum_{i=1}^d p_i=1}$ for the probability simplex and $\triangle_d^{\le}\coloneqq\Set{p\in[0, 1]^d}{\sum_{i=1}^d p_i\leq 1}$ for the subdistribution simplex. Given a finite set $A$, we write $\distf(A)$ for the set of distributions $\nu\colon A\rightarrow [0, 1]$ over $A$, that is, $\sum_{a\in A} \nu(a) = 1$.
A set $X\subseteq\R_+^d$ is \emph{downward closed}
if $x\in X$ and $0\le y\le x$ imply $y\in X$, where the inequality is coordinatewise.
For $X\subseteq\R_+^d$, its \emph{downward closure} is
\begin{equation}
\dws(X)\coloneqq
\Set{y\in\R_+^d}{y\le x\text{ for some }x\in X}.
\end{equation}
For compact nonempty $A\subseteq\R^d$ and $w\in\R^d$, the \emph{support function} is $h_A\prn{w}\coloneqq\max_{x\in A}\inpr{w,x}\in\R$;
the \emph{gauge function} is $\gamma_A\prn{x}\coloneqq\inf\Set{\lambda\ge0}{x\in\lambda A}$.
Support functions are convex and positively homogeneous, and $h_{A+B}=h_A+h_B$ holds for the Minkowski sum $A+B$.

\section{Related Work}
\label{sec:related_work}
We consider safety verification through the lens of the \emph{set containment problem}.
Its basic form is to ask whether a convex set $K$ (representing the system's exact behavior) is a subset of another convex set $S$ (describing safety conditions), which has been studied from various angles.
When $K$ is completely known, the containment test is relatively easy.
Specifically, general polyhedral or quadratic set containment is well-characterized via duality~\citep{mangasarian2002set,jeyakumar2003characterizing}, and its computational complexity depends on the structures of $K$ and $S$~\citep{gritzmann1994complexity}.

We usually need to estimate system's behavior $K$.
While convex set estimation is common in statistical learning~\citep{prince1990reconstructing,Guntuboyina2012AOS,Hannah2013JMLR}, it does not guarantee \emph{soundness}, namely, the estimated $\hat K$ lacks the guarantee $K\subseteq\hat K$, which is crucial for safety verification.
The specificity of safety verification is the downward-closedness assumptions on $K$ and $S$ (see \cref{definition:verification}), which enables soundness guarantee.
In this line, sandwich algorithms for convex multi-objective optimization~\citep{solanki1993approximating,RennenDH11,lammel2023convergence} are well-known: therein, we iteratively approximate a Pareto front (or the behavior set $K$) by maintaining polyhedral inner and outer approximations.
At each iteration, we solve a linearly scalarized problem, which yields a new supporting half-space to minimize the Hausdorff distance between the inner and outer approximations.
Sandwich algorithms have been used for multi-objective probabilistic
model checking of MDPs~\citep{forejt2012pareto,Quatmann23,HartmannsQV26}. %

Whereas sandwich algorithms are common, it does not scale to higher dimensions, limited up to $2$ or $3$ dims.
More severely, they estimate a Pareto front without leveraging structures of the safe region $S$.
\citet{watanabe2026posterior} solve these issues based on a CEGAR framework~\citep{BallR02,ClarkeGJLV03} for safety verification of probabilistic programs.
In particular, they refine the outer approximation by selecting a query with the largest Euclidean violation, effectively leveraging structures of the safe region $S$. %
This is the closest work to ours, and we generalize their query selection from the convex-analytic viewpoint.

\section{Formal Verification of Probabilistic Systems}
\label{sec:mdp-setup}

Our study is motivated by verification of \emph{Markov decision processes (MDPs)}~\citep{Puterman94}.
While the set containment problem we study belongs to computational geometry and can be considered independently of formal verification, we summarize the context of formal verification to help the ML audience better understand our motivation.
We specifically review the probabilistic model checking algorithm developed by \citet{watanabe2026posterior}.
Readers may safely jump to \cref{sec:setup} if they are sufficiently motivated with the set containment problem.

\subsection{Review of MDPs in Formal Verification}

We recall basics of MDPs for formal verification; see \citet{baier2008principles} for probabilistic model checking of MDPs. A (fully observable, discrete-time, finite-state) \emph{MDP} is a tuple $\mdp \coloneqq (Q, q_{\mathtt{i}}, T, A, P)$, where (i) $Q$ is a finite \emph{state} set, (ii) $q_{\mathtt{i}} \in Q$ is the \emph{initial state}, (iii) $T \subseteq Q$ is the set of \emph{target states} with $d \coloneqq |T|$, (iv) $A$ is a finite \emph{action} set, and (v) $P\colon Q \times A \times Q \rightarrow [0,1]$ is a \emph{transition function} satisfying the following conditions:
\begin{itemize}
    \item $\sum_{q' \in Q} P(q,a,q') \in \set{0,1}$ holds ($\forall (q,a)\in Q\times A$).%
    \item Every target state is a sink state: for every $t \in T$ and $a \in A$, we have $\sum_{q' \in Q} P(t,a,q') = 0$. 
\end{itemize}

In this work, we focus on \emph{reachability probabilities} for individual target states $t \in T$, instead of assigning a reward to each state.
To this end, we introduce the reachability of MDP states.
A \emph{scheduler} (or \emph{policy}) is a function $\sigma\colon Q\times (A\times Q)^{\ast}\rightarrow \distf(A)$. We write $\Theta$ for the set of schedulers.
For each history $h$, let $\operatorname{last}(h)$ denote its final state.
A scheduler $\sigma$ for an MDP $\mdp$ induces a \emph{discrete-time Markov chain} $\mc{\sigma}$ with a countable state space, defined as follows:
\begin{itemize}
    \item The state space is the set of histories $H\coloneqq Q\times (A\times Q)^{\ast}$, and the initial state is $q_{\mathtt{i}}$. 
    \item The transition function $P^{\sigma}\colon H\times H\rightarrow [0,1]$ is
    \[
P^{\sigma}(h_1,h_2)\coloneqq
    \sigma(h_1)(a)\cdot P(\operatorname{last}(h_1),a,q),
    \]
    if $h_2=h_1\cdot(a,q)$
              for some $(a,q)\in A\times Q$, and 
    $P^{\sigma}(h_1,h_2)\coloneqq 0$ otherwise.         
    Note that $P$ is the transition function of the original MDP $\mdp$.
\end{itemize}

Given scheduler $\sigma\in \Theta$ and a target state $t\in T$, 
a \emph{path} $\xi$ on $\mc{\sigma}$ to $t$ is a finite sequence $(h_1,\cdots, h_n)$ of histories such that $h_1 = q_{\mathtt{i}}$ and $\operatorname{last}(h_n) = t$. We write $\pathmc{\sigma}{t}$ for the set of paths to $t$.  Given a path $\xi\in \pathmc{\sigma}{t}$, we define its \emph{path probability} $P^{\sigma}(\xi)$ by $P^{\sigma}(\xi)\coloneqq \prod^{n-1}_{i=1} P^{\sigma}(h_i, h_{i+1})$, where $\xi = (h_1,\cdots,h_n)$. 
The \emph{reachability probability} $\rp{\sigma}{t}$ to a target state $t$ under a scheduler $\sigma$ is given by%
\[
\rp{\sigma}{t}\coloneqq \sum_{\xi\in \pathmc{\sigma}{t}} P^{\sigma}(\xi). 
\]
Finally, we introduce the reachability objective.
\begin{definition}[reachability objective]
 Given an MDP $\mdp$ and a function $w\in \distf(T)$,
 the \emph{maximum weighted reachability} $\wrp{w}\in [0, 1]$ is defined by%
 \footnote{
   For any rational-valued $w$,
   $\wrp{w}$ can be computed in polynomial time by solving
   a linear program~\citep{baier2008principles}.
   Moreover, this value is attained by a deterministic
   memoryless scheduler.
 }
\[
     \wrp{w} = \sup_{\sigma\in \Theta}\sum_{t\in T} w(t)\cdot \rp{\sigma}{t}. 
\]
 \end{definition}

Although linear programming computes $\wrp{w}$ in polynomial time, \emph{value iteration}~\citep{ChatterjeeH08} is often efficient in practice and can be extended to compute both lower and upper bounds on $\wrp{w}$~\citep{HaddadM18,QuatmannK18,HartmannsK20,WatanabeVJH24}.
\begin{definition}
    An \emph{upper bound} $u$ of $\wrp{w}$ is a scalar $u\in[0, 1]$ such that $\wrp{w} \leq u$.
    Writing $T=\set{t_1,\dots,t_d}$, a \emph{componentwise lower bound} $l$ of $\wrp{w}$ is a vector $l\in [0, 1]^d$ such that there exists a scheduler $\sigma\in\Theta$ satisfying $l_i\le\rp{\sigma}{t_i}$ for $\forall i\in[d]$.
\end{definition}
Value iteration requires exponentially many iterations in the worst case~\citep{HaddadM18}, but it often outperforms general-purpose LP solvers on common verification benchmarks~\citep{HartmannsJQW23}.\looseness=-1

In what follows, we identify $w\in\distf(T)$ with the vector in $\triangle_d$ having coordinates $w_i\coloneqq w(t_i)$ for $i=1,\dots,d$.

\subsection{CEGAR Loop for Safety Verification}

\cite{watanabe2026posterior} propose the following safety verification problem of MDPs:\footnote{Precisely,~\cite{watanabe2026posterior} study a safety verification problem of probabilistic programs,  which can be seen as MDPs with countably infinite states. In this paper, we restrict our attention to finite-state MDPs, which is still a reasonably expressive class of real-world products~\citep{baier2008principles,ForejtKNP11}, in order to implement an efficient oracle on top of Storm~\citep{HenselJKQV22}. } 
\begin{definition}[safety verification]
\label{definition:verification}

Given an MDP $\mdp$ with the target states $T = \{t_1,\cdots, t_d\}$ and a safe region $S\subseteq [0, 1]^d$ (that is nonempty closed convex, and moreover, downward closed), \emph{safety verification} asks if the following set containment holds:
\[
    \big(\rp{\sigma}{t_1},\dots, \rp{\sigma}{t_d}\big)\in S \quad \forall \sigma\in\Theta. 
\]  
We say that $\mdp$ is \emph{safe} (w.r.t.~$S$) if the above set containment holds, and otherwise \emph{unsafe}. 
\end{definition}

Below we present an illustrative verification example.

\begin{figure}[t]
\begin{center}
\scalebox{0.8}{
\begin{tikzpicture}[
    >={Latex[length=2mm]},
    font=\small,
    state/.style={circle,draw,thick,minimum size=8mm,inner sep=1pt},
    target/.style={state,double,double distance=1pt},
    branch/.style={circle,fill,inner sep=1.6pt},
    transition/.style={->,thick},
    probability/.style={fill=white,inner sep=1.4pt},
    description/.style={font=\scriptsize,text=black!75}
]
    \node[state] (init) at (0,2.85) {$q_{\mathtt{i}}$};
    \node[branch] (a) at (-2.55,0) {};
    \node[branch] (b) at ( 2.55,0) {};
    \node[target] (t1) at (0, 1.8) {$t_1$};
    \node[state]  (ok) at (0, 0) {$q_{\mathrm{ok}}$};
    \node[target] (t2) at (0,-1.1) {$t_2$};
    \draw[transition] (0,4.0) -- (init);
    \draw[transition] (init.west)
        .. controls (-2.55,2.85) and (-2.55,1.7) .. (a)
        node[pos=.58,left,align=right] {$a_1$};
    \draw[transition] (init.east)
        .. controls (2.55,2.85) and (2.55,1.7) .. (b)
        node[pos=.58,right,align=left] {$a_2$};
    \draw[transition] (a) -- node[probability,above,sloped] {$0.005$} (t1);
    \draw[transition] (a) -- node[probability,above]        {$0.99$}  (ok);
    \draw[transition] (a) -- node[probability,below,sloped] {$0.005$} (t2);
    \draw[transition] (b) -- node[probability,above,sloped] {$0.009$} (t1);
    \draw[transition] (b) -- node[probability,above]        {$0.99$}  (ok);
    \draw[transition] (b) -- node[probability,below,sloped] {$0.001$} (t2);
    \node[description,above] at (0,2.14) {};
    \node[description,below] at (0,-0.45) {};
    \node[description,below] at (0,-1.94) {};
    \path (ok) edge [loop above] node {1.0} (ok);
\end{tikzpicture}
}
\vspace{-25pt}
\caption{
  An example MDP representing a toy probabilistic system.
  Here $t_1$ and $t_2$ are failure states.
}
\label{fig:exMDP}
\end{center}
\end{figure}
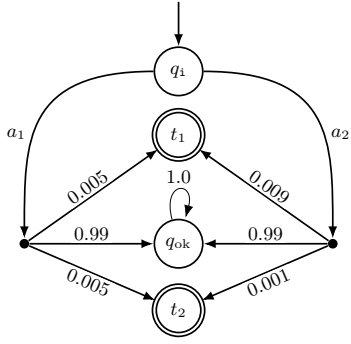

\begin{example}
Consider the MDP in~\cref{fig:exMDP}, with state set
$Q\coloneqq\{q_{\mathtt{i}},q_{\mathrm{ok}},t_1,t_2\}$.
The transition probabilities for action $a_1$ from the initial state
$q_{\mathtt{i}}$ are given by
$P(q_{\mathtt{i}},a_1,t_1) \coloneqq 0.005$, $P(q_{\mathtt{i}},a_1,q_{\mathrm{ok}}) \coloneqq 0.99$, and $P(q_{\mathtt{i}},a_1,t_2) \coloneqq 0.005$.
The corresponding probabilities for action $a_2$ are
$0.009$, $0.99$, and $0.001$, respectively.
The target states $t_1$ and $t_2$ represent two undesirable outcomes.
Actions are selected by, e.g., user inputs.

Suppose that we execute the system twice under the same
fixed environmental choice with independent randomness.
We require the probability that both executions reach the same target state to be at most a threshold $\theta>0$,
regardless of user inputs.
This requirement is equivalent to safety verification
with the safe region
    $S\coloneqq
    \left\{
        (p_1,p_2)\in[0,1]^2
        \;\middle|\;
        p_1^2+p_2^2\leq\theta
    \right\}$.
\end{example}

We review a CEGAR-based solution to safety verification proposed by \cite{watanabe2026posterior}.
Given an MDP $\mdp$, 
we set the \emph{achievable set} 
\[
K_{\mdp} \coloneqq \dws(\{(\rp{\sigma}{t_1},\dots, \rp{\sigma}{t_d})\in[0, 1]^d\mid \sigma\in \Theta \}).
\]
The achievable set $K_{\mdp}$ is a convex polytope~\citep{etessami2008multi}.
Since the safe region $S$ is downward-closed, safety verification is equivalent to checking the inclusion $K_{\mdp}\subseteq S$.
It is possible to test all vertices of the polytope $K_{\mdp}$
to conclude $K_{\mdp}\subseteq S$, but intractable due to the
existence of exponentially many deterministic memoryless schedulers.

\begin{algorithm}[t]
\caption{A CEGAR-Based Safety Verification Algorithm~\citep{watanabe2026posterior}}
\label{alg:cegarloop}
\begin{algorithmic}[1]
\Require An MDP $\mdp$ and a safe region $S$.
\State $U_0\gets\triangle_d^{\le}$; $t\gets0$
\While{$U_t\not\subseteq S$}
  \State $w_t\gets\Call{Select}{U_t,S}$
  \State $(u,l)\gets\Call{Query}{\mdp,w_t}$
    \If{$l\notin S$}
      \State \Return $\UNSAFE$ (meaning that $K_{\mdp}\not\subseteq S$)
    \Else
      \State $U_{t+1}\gets U_t\cap\Set{p}{\inpr{w_t,p}\le u}$
    \EndIf
  \State $t\gets t+1$
\EndWhile
\State \Return \SAFE\, (meaning that $K_{\mdp}\subseteq S$)
\end{algorithmic}
\end{algorithm}

To bypass the enumeration of schedulers,
the CEGAR-based algorithm (\cref{alg:cegarloop})
maintains an outer approximation $U_t$ of
$K_{\mdp}$ at each $t$-th refinement step, ensuring that
$K_{\mdp}\subseteq U_t$.%
\footnote{
The CEGAR loop of~\cite{watanabe2026posterior} is designed
for infinite-state MDPs and does not report $\UNSAFE$.
Extending it to allow this verdict for finite-state MDPs,
as we do here, is straightforward.
}
If $U_t\subseteq S$, we conclude that $\mdp$ is safe.
Since $U_t$ is a convex polytope by its construction and $S$ is convex,
this inclusion can be checked by testing whether every
vertex of $U_t$ lies in $S$.

The key step is the refinement procedure.
When $U_t\not\subseteq S$, the algorithm selects a weight
$w_t$ according to a specified criterion (line~3).
It then computes an upper bound $u\in[0,1]$
and a componentwise lower bound $l\in[0,1]^d$ of  $\wrp{w_t}$. 
These can be computed using probabilistic model checkers such as Storm~\citep{HenselJKQV22}.
If $l\notin S$, the algorithm returns $\UNSAFE$
(lines~5--6); otherwise, it refines the outer approximation
$U_t$ by $u$ (line~8).
While an $L^2$-projection refinement has been proposed~\citep{watanabe2026posterior},
we will see that the performance of~\cref{alg:cegarloop} depends strongly on weight selection of line~3.
We generalize this weight selection step in \cref{sec:scores}.

\section{Problem Setup}
\label{sec:setup}

Motivated by safety verification in \cref{sec:mdp-setup}, here is our primary problem setup in this paper.
For an unknown nonempty compact convex set
$K\subseteq\triangle_d^{\le}$ that is downward closed,
our goal is to test $K\subseteq S$, where
$S\subseteq\triangle_d^{\le}$ is a nonempty compact convex
downward-closed safe region.
Despite the lack of a priori access to $K$,
we can \emph{query} in direction $w\in\triangle_d$ to ask an oracle, who returns a scalar $u$ satisfying $h_K(w)\le u$.
Then, we maintain an \emph{outer approximation} $U\subseteq\triangle_d^\le$ satisfying $K\subseteq U$.
Beginning with $U_0\coloneqq\triangle_d^\le$, we refine it at every step by
\begin{equation}
    \label{eq:update}
    U_{t+1} \coloneqq U_t\cap\Set{x\in\R^d}{\inpr{w_t,x}\le u_t},
\end{equation}
where $(w_t,u_t)$ is the query-oracle pair at step $t$.
This refinement maintains the outer approximation $K\subseteq U_t$ for all $t$, and we halt the refinement loop once $U_t\subseteq S$ is met, where the safety guarantee $K\subseteq S$ is obtained.
The main focus in this paper is to synthesize an efficient query direction $w$.
The following geometric properties are assumed hereafter.
\begin{assumption}
\label{ass:sets}
The safe region $S\subseteq\triangle_d^{\le}$ is nonempty, compact, convex, and downward closed. The outer approximation $U\subseteq\triangle_d^{\le}$ is nonempty, compact, and convex.
\end{assumption}

\section{Unified Query Selection}
\label{sec:scores}

To test the set containment, we use support functions extensively.
For nonempty compact convex $U,S\subseteq\R^d$,
the containment $U\subseteq S$ is equivalent to $h_U(w)\le h_S(w)$ for all $w\in\R^d$~\citep[Section 13]{rockafellar1970convex}.
Under our downward-closed assumption (\cref{ass:sets}), we need to only consider the \emph{normalized} direction $w\in\triangle_d$ to determine the set containment.

\begin{lemma}
\label{lem:containment}
Let $B\subseteq\R^d_+$ be nonempty, compact, convex, and downward closed in $\R^d_+$, and let $U\subseteq\R^d_+$ be nonempty and compact. Then $U\subseteq B$ holds if and only if $h_U\prn{w}\le h_B\prn{w}$ holds for all $w\in\triangle_d$.
\end{lemma}
\begin{proof}
The forward implication follows from the definition of the support function. For the reverse implication, let $x\in U\setminus B$ and let $q$ be its Euclidean nearest point in $B$. We have $q\le x$: lowering any coordinate of $q$ that exceeds the corresponding coordinate of $x$ preserves membership in $B$ and strictly decreases the distance. The projection inequality gives $\inpr{x-q,b-q}\le0$ for all $b\in B$. Thus $x-q$ is nonzero and nonnegative, and $w\coloneqq\prn{x-q}/\norm{x-q}_1$ satisfies $h_B\prn{w}=\inpr{w,q}<\inpr{w,x}\le h_U\prn{w}$,
where the strict inequality follows by $x\ne q$.
This is a contradiction.
\end{proof}

By \cref{lem:containment}, a positive \emph{support gap} $h_U\prn{w}-h_S\prn{w}$ identifies a direction $w\in\triangle_d$ in which the outer approximation $U$ violates the safe region $S$.
Subsequently, we introduce the normalized support gap to obtain a scale-invariant score for comparing violations across directions in a chosen geometry.

\subsection{Direction Score}
\label{sec:scores-def}

To take the underlying geometry into account, we introduce an
\emph{error unit} $E\subseteq\R_+^d$, which is a nonempty compact
convex downward-closed set and satisfies
$h_E>0$ over $\triangle_d$.
For $\epsilon\ge0$, we enlarge $S$ to $S+\epsilon E$ by allowing additive error vectors from the scaled error unit $\epsilon E$.
The enlarged set $S+\epsilon E$ remains nonempty, compact, convex, and downward closed, which matters on applying \cref{lem:containment}.
By additivity and homogeneity of support functions, we have
$h_{S+\epsilon E}=h_S+\epsilon h_E$.
Users are allowed to choose any $E$ suitable for their purposes.
We study the following three choices:
\begin{equation}
E_2\coloneqq\Set{z\!\in\!\R^d_+\!}{\!\norm{z}_2\!\le\!1},\;\; E_\infty\coloneqq\brc{0,1}^d,\;\; E=S.
\label{eq:three-units}
\end{equation}
The first two measure errors with respect to the Euclidean and $L^\infty$ geometries, respectively.
The third choice respects the geometry of the safe region $S$ itself, leading to the proportional enlargement $S+\epsilon S=(1+\epsilon)S$ thanks to the convexity of $S$.
The last one requires the safe region $S$ to satisfy $h_S>0$.
\cref{tab:units} summarizes the three choices of the error unit.

\begin{table}[t]
\caption{Error units and normalizers for $w\in\triangle_d$.
The gauge choice ($E=S$) requires $h_S>0$ on this simplex.}
\label{tab:units}
\begin{center}
\begin{tabular}{@{}lll@{}}
\toprule
Rule & Error unit $E$ & $h_E\prn{w}$\\
\midrule
$L^2$ & $E_2$ in \eqref{eq:three-units} & $\norm{w}_2$\\
$L^\infty$ & $E_\infty=\brc{0,1}^d$ & $\norm{w}_1=1$\\
gauge & $S$ & $h_S\prn{w}$\\
\bottomrule
\end{tabular}
\end{center}
\end{table}

Now, we introduce the following \emph{direction score}, which is the normalized support gap, and its \emph{excess}:
\begin{equation}
\begin{aligned}
\textbf{Direction score} \quad &
v_E\prn{w}\coloneqq\frac{h_U\prn{w}-h_S\prn{w}}{h_E\prn{w}}, \\
\textbf{Excess} \quad & V_E\prn{U,S}\coloneqq\pos{\max_{w\in\triangle_d}v_E\prn{w}}.
\end{aligned}
\label{eq:score}
\end{equation}
Due to the homogeneity of support functions, the scale of $w$ no longer affects the direction score $v_E(w)$.
While the positive $v_E(w)$ means the violation $U \nsubseteq S$ in this direction $w$, the excess measures the minimally necessary enlargement $\epsilon E$ to uniformly ensure $U\subseteq S+\epsilon E$.

\begin{theorem}
\label{thm:containment}
Under \cref{ass:sets}, let $E\subseteq\R_+^d$ be a given error unit that is nonempty, compact, convex, downward closed, and satisfies $h_E>0$ over $\triangle_d$.
For every $\epsilon\ge0$, we have $U\subseteq S+\epsilon E$ if and only if \mbox{$v_E\prn{w}\le\epsilon$} holds for all $w\in\triangle_d$. Consequently, we have
\begin{equation}
V_E\prn{U,S}=\min\Set{\epsilon\ge0}{U\subseteq S+\epsilon E}.
\label{eq:containment}
\end{equation}
\end{theorem}
\begin{proof}
Apply \cref{lem:containment} to $B=S+\epsilon E$. By $h_B=h_S+\epsilon h_E$, the containment $U\subseteq S+\epsilon E$ is equivalent to $h_U\le h_S+\epsilon h_E$ over $\triangle_d$. Dividing by $h_E(w)$ gives $v_E\prn{w}\le\epsilon$. Minimizing over $\epsilon\ge0$ gives \eqref{eq:containment}.
\end{proof}

The largest direction score thus selects the direction in which the outer approximation $U$ violates the safe region $S$ most relative to the error unit $E$.
The excess is zero exactly when the safety condition $U\subseteq S$ holds.

\begin{remark}
\label{rem:infconv}
From the convex-analytic viewpoint, the excess generalizes the directed Hausdorff distance from $U$ to $S$.
To see this, we introduce a distance function
\[
  d_E(x) \coloneqq (\chi_S \infconv \gamma_E)(x) = \min_{q \in S}\gamma_E(x-q),
\]
which is the \emph{infimal convolution} between the indicator function of $S$, $\chi_S$, and the gauge function of $E$.
Then $S+\epsilon E=\Set{x}{d_E\prn{x}\le\epsilon}$ and $V_E\prn{U,S}=\max_{x\in U}d_E\prn{x}$ (\cref{prop:infconv} in \cref{app:proofs}).
As seen later, $d_{E_2}$ and $d_{E_\infty}$ correspond to the $L^2$ and $L^\infty$ point-to-set distances, respectively.
Hence $V_E$ is exactly the directed Hausdorff distance when $E=E_2$.%
\footnote{
  A similar notion to the direction score (called the quality indicator) is also used to measure the approximation quality of sandwich algorithms.
  The quality indicator is equivalent to the directed Hausdorff distance but only up to multiplicative factors~\citep{kamenev1994analysis,lammel2023convergence},
  which is weaker than the exact correspondence of $V_{E_p}$ and the directed Hausdorff distance.
}

This idea of the set distance by an error unit dates back to \citet{Hormander1955}.
For safety verification, we
(i) focus on downward-closed sets and
(ii) allow asymmetric error units $E$.
These points confine queries to $\triangle_d$ (\cref{lem:containment}), enabling query direction synthesis.
\end{remark}

\begin{figure*}[t]
\centering
\begingroup
\definecolor{exteal}{HTML}{008985}
\definecolor{exviolet}{HTML}{7771CE}
\begin{tikzpicture}[
  x=3.4cm, y=3.4cm,
  font=\small,
  line cap=round, line join=round,
  axis/.style={draw=black!45, line width=0.4pt,
    -{Latex[length=1.5mm,width=1mm]}},
  enlargement/.style={draw=exviolet, line width=0.9pt},
  outer/.style={draw=black!65, fill=black!6, line width=0.7pt},
  safe/.style={draw=black!85, fill=black!30, line width=0.9pt},
  direction/.style={draw=exteal, line width=1.4pt,
    -{Latex[length=2.3mm,width=1.55mm]}},
  residual/.style={draw=black!45, line width=0.65pt,
    dash pattern=on 2.2pt off 2pt},
  point/.style={circle, fill=exteal, inner sep=1.55pt},
  every node/.style={inner sep=1.5pt}
]
\pgfmathsetmacro{\exradius}{sqrt(5)/4}
\pgfmathsetmacro{\exdiag}{sqrt(5/32)}

\foreach \panel/\offset/\heading in {
  1/0/{$L^2$},
  2/1.33/{$L^\infty$},
  3/2.66/{Gauge}
}{
  \begin{scope}[shift={(\offset,0)}]
    \ifnum\panel=1
      \path[fill=exviolet!10]
        (0,0) -- ({0.25+\exradius},0)
        arc[start angle=0,end angle=45,radius=\exradius]
        -- (\exdiag,{0.25+\exdiag})
        arc[start angle=45,end angle=90,radius=\exradius]
        -- cycle;
    \else\ifnum\panel=2
      \path[fill=exviolet!10]
        (0,0) -- (0.75,0) -- (0.75,0.5)
        -- (0.5,0.75) -- (0,0.75) -- cycle;
    \else
      \path[fill=exviolet!10]
        (0,0) -- (1,0) -- (0,1) -- cycle;
    \fi\fi

    \path[outer] (0,0) rectangle (0.75,0.25);
    \path[safe] (0,0) -- (0.25,0) -- (0,0.25) -- cycle;

    \ifnum\panel=1
      \draw[enlargement] ({0.25+\exradius},0)
        arc[start angle=0,end angle=45,radius=\exradius]
        -- (\exdiag,{0.25+\exdiag})
        arc[start angle=45,end angle=90,radius=\exradius];
    \else\ifnum\panel=2
      \draw[enlargement] (0.75,0) -- (0.75,0.5)
        -- (0.5,0.75) -- (0,0.75);
    \else
      \draw[enlargement] (1,0) -- (0,1);
    \fi\fi

    \draw[axis] (0,0) -- (1.07,0) node[right] {$x_1$};
    \draw[axis] (0,0) -- (0,1.08) node[above] {$x_2$};
    \node[below left, text=black!65] at (0,0) {$0$};
    \draw[black!45,line width=0.4pt] (1,0) -- (1,-0.015);
    \node[below, text=black!65] at (1,-0.018) {$1$};
    \draw[black!45,line width=0.4pt] (0,1) -- (-0.015,1);
    \node[left, text=black!65] at (-0.018,1) {$1$};

    \node at (0.073,0.074) {$S$};
    \node[text=black!70] at (0.48,0.18) {$U$};

    \ifnum\panel<3
      \draw[residual] (0.25,0) -- (0.75,0.25);
      \node[point] at (0.25,0) {};
      \node[below, text=exteal] at (0.25,-0.018) {$q_*$};
    \fi

    \ifnum\panel=1
      \draw[direction] (0.75,0.25)
        -- ({0.75+0.26*2/sqrt(5)},{0.25+0.26/sqrt(5)})
        node[above right, inner sep=0.3pt, xshift=-1pt, yshift=-1pt, text=exteal] {$w_*$};
    \else\ifnum\panel=2
      \draw[direction] (0.75,0.25) -- (1.01,0.25)
        node[right, inner sep=0.3pt, text=exteal] {$w_*$};
    \else
      \draw[direction] (0.75,0.25)
        -- ({0.75+0.26/sqrt(2)},{0.25+0.26/sqrt(2)})
        node[above right, inner sep=0.3pt, xshift=-1pt, yshift=-1pt, text=exteal] {$w_*$};
    \fi\fi
    \node[point] at (0.75,0.25) {};
    \node[above left, text=exteal] at (0.695,0.26) {$x_*$};
    \node[font=\normalsize] at (0.5,-0.18) {\heading};
  \end{scope}
}
\end{tikzpicture}
\caption{Different error units select different query directions for the
same safe region $S$ (\textcolor{black!85}{\textbf{dark gray}}) and outer approximation $U$
(\textcolor{black!65}{\textbf{light gray}}).
The \textcolor{exviolet}{\textbf{violet}} regions show the smallest enlargements $S+\epsilon E$
containing $U$: $\epsilon=\sqrt5/4$ for $E=E_2$, $1/2$ for
$E=E_\infty$, and $3$ for $E=S$.
At $x_*=(3/4,1/4)$, the \textcolor{exteal}{\textbf{teal}} arrows are outward normals to these
enlargements, giving the normalized weights $w_*=(2/3,1/3)$, $(1,0)$, and
$(1/2,1/2)$.
The dashed segments join $x_*$ to the common nearest point
$q_*=(1/4,0)$ for the $L^2$ and $L^\infty$ cases.}
\label{fig:example-query-directions}
\endgroup
\end{figure*}
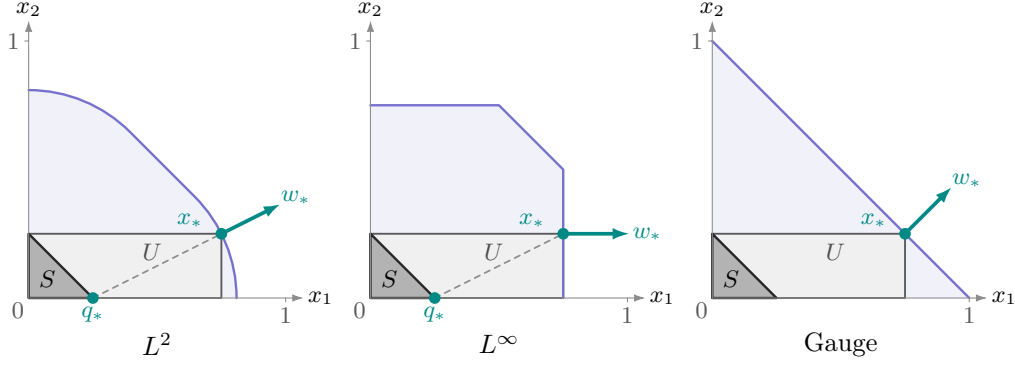

\subsection{Query Computation}
\label{sec:three}

To implement the refinement step~\eqref{eq:update}, we use the maximizer of the direction score $v_E$ as a query.
Since the direction score $v_E$ is not concave, its maximization is generally nontrivial.
For each geometry in \cref{tab:units}, we instead compute the query in the following three steps:
\begin{enumerate}
  \item Compute a ``farthest'' point $x_*\in U$ from $S$;
  \item ``Project'' $x_*$ onto $S$, denoted by $q_*\in S$;
  \item Query a supporting normal $w_*$ at $q_*$.
\end{enumerate}
The nonconvexity remains only in Step~1: in all cases below, the ``distance'' to $S$ is convex in $x$, so its maximization over $U$ is intractable in general.
We typically consider polyhedral $U$ generated by the refinement step~\eqref{eq:update}; then the maximum is attained at a vertex, so enumerating the vertices of $U$ suffices to find $x_*$.

For $p\in\set{2,\infty}$ and $x\in\R^d_+$, let $\dist_p\prn{x,S}\coloneqq\min_{s\in S}\norm{x-s}_p$, and write $\proj_S^p\prn{x}$ for a chosen minimizer of this distance.
All proofs of the formal claims in this subsection are deferred to \cref{app:proofs}.

\subsubsection{\texorpdfstring{$L^2$}{L2} Score}
For $\epsilon\ge0$, $S+\epsilon E_2$ is the closed Euclidean
$\epsilon$-neighborhood of $S$ within $\R^d_+$.
A farthest point outside $S$ has a unique Euclidean projection
thanks to the convexity of $S$, and the normalized projection
residual maximizes $v_{E_2}$.
\begin{proposition}
\label{prop:l2}
Under \cref{ass:sets}, we have $V_{E_2}\prn{U,S}=\max_{x\in U}\dist_2\prn{x,S}$. If this value is positive, let $x_*$ be a (non-unique) maximizer of the distance and let $q_*\coloneqq\proj_S^2\prn{x_*}$. Then
\begin{equation}
w_*\coloneqq\frac{x_*-q_*}{\norm{x_*-q_*}_1}
\label{eq:l2-direction}
\end{equation}
belongs to $\triangle_d$ and maximizes $v_{E_2}$.
\end{proposition}
Finding $x_*$ by vertex enumeration of $U$ recovers the refinement step of \citet[Algorithm~2]{watanabe2026posterior}.
After getting a farthest point $x_*$, the query $w_*$ in
\eqref{eq:l2-direction} is computed by normalizing the projection residual.

\subsubsection{\texorpdfstring{$L^\infty$}{L infty} Score}
The set $S+\epsilon\brc{0,1}^d$ allows an increase of at most $\epsilon$ in each coordinate of a safe vector.
With this choice $E_\infty=\brc{0,1}^d$, we have $h_E(w)=\norm{w}_1 = 1$, which implies $v_{E_\infty}(w) = h_U(w) - h_S(w)$ for all $w\in\triangle_d$.
Thus, the $L^\infty$ unit measures the set containment by the unnormalized support gap, which is the absolute error of the support functions.
Below, we characterize maximizers of $v_{E_\infty}$.
\begin{proposition}
\label{prop:linf}
Under \cref{ass:sets}, we have $V_{E_\infty}\prn{U,S}=\max_{x\in U}\dist_\infty\prn{x,S}$. If this value is positive, let $x_*$ be a (non-unique) distance maximizer, $\delta\coloneqq\dist_\infty\prn{x_*,S}$, and $q_*\coloneqq\pos{x_* - \delta\mathbf1}$, where $\mathbf1$ is the all-ones vector. Then $q_*$ is an $L^\infty$-nearest point in $S$ to $x_*$. 
There exists $w_*\in\triangle_d$ satisfying
\begin{equation}
h_S\prn{w_*}=\inpr{w_*,q_*},\qquad
\inpr{w_*,x_*-q_*}=\delta,
\label{eq:linf-conditions}
\end{equation}
and every such $w_*$ maximizes $v_{E_\infty}$.
\end{proposition}
Unlike the $L^2$ case, \cref{prop:linf} gives the existence of $w_*$ only.
Once $x_*$ is obtained by vertex enumeration, $q_*$ follows immediately.
It remains to construct $w_*$.
Let
\[
I\coloneqq\Set{i\in[d]}{x_{*,i}-q_{*,i}=\delta}
\]
be the saturated coordinates,
and let $\mathbf1_I$ be the indicator vector of $I$.
The first condition in \eqref{eq:linf-conditions}
says that $w_*$ is an outward normal of $S$ at $q_*$,
and the second says that $w_*$ puts weight only on $I$
(\cref{lem:linf-saturated}).
For $i\notin I$, we have $q_{*,i}=0$,
so the nonnegativity constraint $x_i\ge0$ is active
and can cancel the weight on such coordinates.
Hence, a nonnegative outward normal can be \emph{masked} by $\mathbf1_I$,
i.e., its coordinates outside $I$ can be set to zero, while it remains an outward normal.
When $S$ is given by a convex function,
this yields the following closed-form $w_*$.\looseness=-1

\begin{proposition}
\label{prop:linf-smooth}
In the setting of \cref{prop:linf},
let $S=\Set{x\in\triangle_d^\le}{f(x)\le\theta}$
for a differentiable convex $f$ that is
nondecreasing in each coordinate
and a threshold $\theta\in\R$ with $f(0)<\theta$.
Then $\nabla f(q_*)\odot\mathbf1_I\neq0$, and
$
  w_*\coloneqq
  {\nabla f(q_*)\odot\mathbf1_I}/{\norm{\nabla f(q_*)\odot\mathbf1_I}_1}
$
satisfies the certificate~\eqref{eq:linf-conditions},
where $\odot$ denotes the coordinatewise product.
\end{proposition}
Thus, $w_*$ is available in closed form as a normalized masked gradient.
Its complexity is subject to the gradient evaluation of $\nabla f(q_*)$.

\subsubsection{Gauge Score}
To use $E=S$ as an error unit, assume
\begin{equation}
h_S\prn{w}>0\quad\text{for all }w\in\triangle_d.
\label{eq:positivity}
\end{equation}
Under \cref{ass:sets}, this is equivalent to $\rho\triangle_d^{\le}\subseteq S$ for some $\rho>0$, i.e., $S$ does not require any failure probability to be exactly zero.
Under the gauge choice $E=S$, we have the proportional enlargement $S+\epsilon S=\prn{1+\epsilon}S$.
Now we see that the gauge excess is related to the gauge function of $S$, $\gamma_S$.

\begin{proposition}
\label{prop:gauge}
Under \cref{ass:sets} and \eqref{eq:positivity},
\begin{equation}
V_S\prn{U,S}=\pos{\max_{x\in U}\gamma_S\prn{x}-1}
\label{eq:gauge-excess}
\end{equation}
holds.
If \eqref{eq:gauge-excess} is positive, let $x_*\in U$ maximize $\gamma_S$, $\lambda_*\coloneqq\gamma_S\prn{x_*}$, and $q_*\coloneqq x_*/\lambda_*$.
Then $q_*$ lies on the boundary of $S$, and every $w_*\in\triangle_d$ satisfying $h_S\prn{w_*}=\inpr{w_*,q_*}$ also satisfies $\inpr{w_*,x_*}=h_U\prn{w_*}$ and maximizes $v_S$.
\end{proposition}
Since the gauge function $\gamma_S$ is convex, its maximization is again intractable.
Thus, we obtain a gauge maximizer $x_*$ by vertex enumeration, as in the $L^2$ and $L^\infty$ cases.
We can compute $w_*$ as a supporting hyperplane of $S$ at $q_*$ when $S$ is a smooth safe region, under the same assumption as \cref{prop:linf-smooth}.
\begin{proposition}
\label{prop:gauge-smooth}
Let $S=\Set{x\in\triangle_d^{\le}}{f\prn{x}\le\theta}$
for a differentiable convex $f$ that is nondecreasing in each coordinate
and a threshold $\theta\in\R$ with $f\prn{0}<\theta$.
Then \eqref{eq:positivity} holds.
Moreover, in the setting of \cref{prop:gauge}, we have $f\prn{q_*}=\theta$, and if $\nabla f\prn{q_*}\neq0$, then
$
  w_*\coloneqq
  {\nabla f\prn{q_*}}/{\norm{\nabla f\prn{q_*}}_1}
$
belongs to $\triangle_d$ and maximizes $v_S$.
\end{proposition}
Unlike the $L^\infty$ case (\cref{prop:linf-smooth}), the gradient masking is unnecessary:
every supporting normal at $q_*$ is optimal by \cref{prop:gauge},
and $q_*$ lies off the face $\inpr{\mathbf1,x}=1$ since $\inpr{\mathbf1,q_*}\le1/\lambda_*<1$.
Here, $\lambda_*=\min\Set{\lambda\ge1}{f\prn{x_*/\lambda}\le\theta}$ is obtained by bisection.

\begin{remark}
The pair $\prn{x_*,w_*}$ in \cref{prop:gauge} admits a dual interpretation via the \emph{polar set} of $S$:
$$S^\circ\coloneqq\Set{y\in\R^d}{\inpr{y,x}\le 1\;\forall x\in S}.$$
Since $0\in S$ by downward closedness, we have $\gamma_S=h_{S^\circ}$ by the duality~\citep[Theorem~14.5]{rockafellar1970convex}, and hence $\max_{x\in U}\gamma_S\prn{x}=\max_{y\in S^\circ}h_U\prn{y}$.
In \cref{prop:gauge}, $y_*\coloneqq w_*/h_S\prn{w_*}$ belongs to $S^\circ$ and satisfies $\inpr{y_*,x_*}=h_U\prn{y_*}=h_{S^\circ}\prn{x_*}$.
That is, $x_*$ is a supporting point of $U$ in direction $y_*$, and $y_*$ is a supporting point of $S^\circ$ in direction $x_*$; the query $w_*$ is the $\ell_1$-normalization of the latter.
\end{remark}

\begin{table*}[t]
\centering
\caption{All 24 model/specification pairs, grouped by expected label (12 cases per group). D and U count decisive and \UNKNOWN{} results. $N,W,T$ sum queries, query selection time, and runtime over all 12 runs in each row.}
\label{tab:exp-natural}
\small
\setlength{\tabcolsep}{8pt}
\begin{center}
\begin{tabular}{@{}llrrrrr@{}}
\toprule
Expected & Method & D ($\uparrow$) & U ($\downarrow$)
& $N$ ($\downarrow$) & $W$ (sec) & $T$ (sec) \\
\midrule
\SAFE
& $L^2$~\citep{watanabe2026posterior}
& 10 & 2 & 247 & 3044.879 & 3316.6 \\
& $L^\infty$
& 9 & 3 & 284 & 606.262 & 893.7 \\
& Gauge
& 12 & 0 & 78 & 39.046 & 105.3 \\
& LK-$\epsilon$~\citep{lammel2023convergence}
& 5 & 7 & 441 & 4242.446 & 4880.7 \\
& LK-$\gamma$~\citep{lammel2023convergence}
& 9 & 3 & 312 & 3608.969 & 4161.0 \\
\addlinespace
\UNSAFE
& $L^2$~\citep{watanabe2026posterior}
& 12 & 0 & 58 & 36.112 & 88.7 \\
& $L^\infty$
& 12 & 0 & 58 & 11.148 & 63.1 \\
& Gauge
& 12 & 0 & 24 & 8.633 & 35.0 \\
& LK-$\epsilon$~\citep{lammel2023convergence}
& 9 & 3 & 209 & 1569.323 & 1949.5 \\
& LK-$\gamma$~\citep{lammel2023convergence}
& 9 & 3 & 190 & 1528.724 & 1895.5 \\
\bottomrule
\end{tabular}
\end{center}
\end{table*}

\subsubsection{Comparison of the Three Rules}
Computationally, the three rules differ in the projection $q_*$.
The $L^2$ rule requires a $d$-dimensional Euclidean projection onto $S$, which calls for an iterative convex solver for general $S$.
In contrast, downward closedness reduces the $L^\infty$ distance to the scalar search $\dist_\infty\prn{x,S}=\min\Set{\delta\ge0}{\pos{x-\delta\mathbf1}\in S}$ (\cref{prop:linf}), and positive homogeneity reduces the gauge to $\gamma_S\prn{x}=\inf\Set{\lambda>0}{x/\lambda\in S}$; both are computed by bisection.
This leads to cheaper queries than the $L^2$ rule, as we will see in \cref{sec:experiments}.

\subsection{Example of Selected Queries}
\label{sec:example}

The following example shows how the three error units lead to different
query directions for the same safe region and outer approximation, illustrated in \cref{fig:example-query-directions}.
The score calculations are given in \cref{app:examples}.

\begin{example}
\label{ex:boxes}
Consider $S=\Set{x\in\R_+^2}{x_1+x_2\le1/4}$ and
$U=\brc{0,3/4}\times\brc{0,1/4}$.
The safe region bounds the sum of the two failure probabilities by $1/4$.
The vertex $x_*=\prn{3/4,1/4}$ is farthest from $S$ in Euclidean distance,
with projection $q_*=\prn{1/4,0}$ and residual $\prn{1/2,1/4}$.
Thus the $L^2$ rule selects $w_*=\prn{2/3,1/3}$.
In the $L^\infty$ distance, $x_*$ is again a farthest point,
with $\delta=1/2$ and $q_*=\prn{x_*-\delta\mathbf1}_+$.
The direction $w_*=\prn{1,0}$ satisfies \eqref{eq:linf-conditions}.
For the gauge, $\gamma_S\prn{x}=4\prn{x_1+x_2}$ is maximized at $x_*$ with $\lambda_*=4$, and the gauge rule selects
$w_*=\prn{1/2,1/2}$ supporting $S$ at $q_*=x_*/4$.
\end{example}

\section{Experiments}
\label{sec:experiments}

We evaluate how the choice of error unit affects safety verification by preliminary experiments.\footnote{Our implementation, including benchmarks, is available here: \url{https://github.com/levelfour/SafetyMDP}.}

\paragraph{Methods and benchmarks.}
We implement the $L^2$, $L^\infty$, and gauge rules of~\cref{sec:three}.
The two baselines approximate the achievable set $K_{\mdp}$ by the sandwich algorithm~\citep{forejt2012pareto}, iteratively selecting directions with the epsilon indicator or polyhedral gauge of~\citet{lammel2023convergence}.
We denote them by LK-$\epsilon$ and LK-$\gamma$, respectively, and the detail is summarized in \cref{sec:indicatorEnc}.
All methods share the same oracle on top of Storm~1.13.0~\citep{HenselJKQV22}.

We create three families of MDP benchmarks: gridworld, robot-mission, and
network-routing.
The models are written in the PRISM
language~\citep{KwiatkowskaNP04}.\footnote{
To the best of our knowledge, no existing benchmark suite is available for this safety verification problem.}
Gridworld models an exploration task under uncertainty,
with hazardous cells represented by target states.
Robot-mission models a robot attempting to complete a sequence
of tasks, with potential failures during each task.
Network-routing is a simplified model of packet transmission,
with different transmission failures represented by target states.

For each benchmark family, we construct MDPs with
$d\in\{4,6,8,10\}$ target states to test the scalability in the dimension~$d$.
For each MDP, we specify two safe regions, one expected to yield
a \SAFE{} verdict and the other an \UNSAFE{} verdict.
This gives eight verification instances per family and
24 instances in total, with 12 expected to be \SAFE{} and
12 expected to be \UNSAFE{}.
All safe regions have the form
$S=\Set{x\in\triangle_d^{\le}}{f\prn{x}\le\theta}$,
where $f$ is quadratic and $\theta\in\R$ is a threshold.
Specifically, we use the following common family of convex
quadratic safety specifications across all benchmarks:
\[
f(x)=\frac{1}{2m}\sum_{j=1}^{m}
\biggl[
\biggl(\frac{a_j^\top x}{b_j}\biggr)^2
+\frac{a_j^\top x}{b_j}
\biggr].
\]
For each benchmark,  we set the limit of query selection steps as $50$, and if the algorithm cannot decide within the limit, it returns \UNKNOWN{}. 

\paragraph{Results.}
\Cref{tab:exp-natural} summarizes the results.
The $L^\infty$ and gauge rules mostly outperform the $L^2$ rule in runtime.
In our implementation, the former compute boundary points by bisection and then obtain query directions using
explicit gradient formulas
(\cref{prop:linf-smooth,prop:gauge-smooth}, respectively),
whereas the latter solves a $d$-dimensional convex projection
problem numerically; this difference may explain the performance difference. 
The sandwich baselines, LK-$\epsilon$ and LK-$\gamma$,
select query directions to approximate the achievable set
without using the safe region.
In contrast, the CEGAR methods use the safe region to guide refinement.
This difference may explain why the sandwich baselines often
underperform the CEGAR methods.
Note that \UNSAFE{} cases tend to halt earlier 
than \SAFE{} cases, since it suffices to find a scheduler whose
reachability-probability vector lies outside the safe region. We also observe occasional numerical difficulties during refinement,
specifically in outer-polytope vertex enumeration and
safe-region support computations;
see~\cref{tab:all-natural-query-counts}
in~\cref{app:experimental-details} for details.

\section{Conclusion}
We studied a convex-analytic formulation to synthesize a query to test set containment, in particular for downward-closed sets.
Empirically, the gauge rule exhibits clear benefits over the $L^2$ rule.
We conjecture that the condition number of $S$ affects the query selection, and it is an interesting future work to investigate how it affects the choice of an error unit $E$.

\subsection*{Acknowledgments}
HB is supported by JST PRESTO (JPMJPR24K6) and JST BOOST (JPMJBY24E8).
SS is supported by JST BOOST (JPMJBY24D1).
KW is supported by the JST PRESTO (JPMJPR25KD).

\subsection*{AI use statement}
In this work, we used generative AI tools for generating synthetic data sets; assisting in the proof writing; implementing methods.
We have not used generative AI tools for helping develop theoretical models or conceptual frameworks; formulating mathematical claims; designing or providing feedback on research methodology or experiments; assisting with translation; interpreting results.
Other tasks including proposing hypotheses and supporting qualitative and thematic data analysis are not applicable to this work.
Additionally, we used generative AI tools for identifying relevant literature, creating figures (\cref{fig:example-query-directions}), and drafting the initial version of the paper.
We have reviewed all AI-assisted work by substantially rewriting the initially generated draft by the authors themselves, reorganizing the paper structure, and reviewing the generated programs and datasets.
We take responsibility for the final content of this work, including text, claims, or artifacts produced with the aid of generative AI.

\bibliography{references}

\clearpage
\appendix
\thispagestyle{empty}
\crefalias{section}{appendix}
\crefalias{subsection}{appendix}
\begin{revision}
\section{Proofs}
\label{app:proofs}

\subsection{Auxiliary claims}
Throughout this appendix, for an error unit $E$, we use the distance function,
\[
d_E\prn{x}\coloneqq\prn{\chi_S\infconv\gamma_E}\prn{x}=\min_{q\in S}\gamma_E\prn{x-q},
\qquad x\in\R^d_+,
\]
where $\chi_S$ is the indicator function of $S$ and $\gamma_E$ is the gauge function of $E$.
The following proposition, referred from \cref{rem:infconv}, shows that the excess $V_E(U,S)$ can be regarded as a generalized directed Hausdorff distance.
Later, this yields the three query rules as the Hausdorff distance maximization, combined with \cref{lem:certificate} below.

\begin{proposition}
\label{prop:infconv}
Under \cref{ass:sets}, let $E\subseteq\R_+^d$ be an error unit that is nonempty, compact, convex, downward closed, and satisfies $h_E>0$ over $\triangle_d$. For every $x\in\R^d_+$ and $\epsilon\ge0$, we have $x\in S+\epsilon E$ if and only if $d_E\prn{x}\le\epsilon$, and
\begin{equation}
d_E\prn{x}=\pos{\max_{w\in\triangle_d}\frac{\inpr{w,x}-h_S\prn{w}}{h_E\prn{w}}}.
\label{eq:infconv-dual}
\end{equation}
Consequently, $d_E$ is finite and continuous on $\R^d_+$, and $V_E\prn{U,S}=\max_{x\in U}d_E\prn{x}$.
Moreover, under $h_S>0$ on $\triangle_d$ assumed in \eqref{eq:positivity}, we have $d_S\prn{x}=\pos{\gamma_S\prn{x}-1}$ for every $x\in\R^d_+$.
\end{proposition}
\begin{proof}
To characterize the sublevel of $d_E$, we first characterize the sublevel set of $\gamma_E$.
Since $E$ is nonempty and downward closed, we have $0\in E$, so $\mu E\subseteq\lambda E$ for $0\le\mu\le\lambda$ by convexity. For $\lambda>0$, closedness of $E$ then gives $\gamma_E\prn{z}\le\lambda$ if and only if $z\in\lambda E$; for $\lambda=0$, $\gamma_E\prn{z}\le0$ if and only if $z\in\bigcap_{\lambda>0}\lambda E=\set{0}=0E$ by boundedness of $E$. Thus $\Set{z}{\gamma_E\prn{z}\le\epsilon}=\epsilon E$ for all $\epsilon\ge0$. In particular, $s\mapsto\gamma_E\prn{x-s}$ has closed sublevel sets, so it is lower semicontinuous and attains its minimum (possibly $+\infty$) on the compact set $S$, and $d_E$ is well-defined.

If $d_E\prn{x}\le\epsilon$, a minimizer $s\in S$ satisfies $x-s\in\epsilon E$, so $x\in S+\epsilon E$. Conversely, $x=s+\epsilon t$ with $s\in S$ and $t\in E$ gives $d_E\prn{x}\le\gamma_E\prn{\epsilon t}\le\epsilon$.
This proves the first claim: $x\in S+\epsilon E$ if and only if $d_E\prn{x}\le\epsilon$.
Applying \cref{lem:containment} with $U=\set{x}$ and $B=S+\epsilon E$, together with $h_B=h_S+\epsilon h_E$, shows that $x\in S+\epsilon E$ if and only if $\inpr{w,x}-h_S\prn{w}\le\epsilon h_E\prn{w}$ for all $w\in\triangle_d$, that is, $\epsilon\ge M\prn{x}\coloneqq\max_{w\in\triangle_d}\prn{\inpr{w,x}-h_S\prn{w}}/h_E\prn{w}$.
This maximum exists because the ratio is continuous in $w$ and $\triangle_d$ is compact.
Combined with the first claim, $d_E\prn{x}\le\epsilon$ if and only if $\pos{M\prn{x}}\le\epsilon$ for every $\epsilon\ge0$, so $d_E\prn{x}=\pos{M\prn{x}}$. This $d_E\prn{x}$ is finite.

Next, we show that $d_E$ is continuous.
The support function $h_E$ is continuous because $E$ is compact, so it attains its minimum $\mu$ on the compact set $\triangle_d$, which is positive because of the assumption $h_E>0$.
For each $w\in\triangle_d$ and $x,y\in\R^d_+$, H\"older's inequality gives $\abs{\inpr{w,x}-\inpr{w,y}}\le\norm{x-y}_\infty$, so the function $x\mapsto\prn{\inpr{w,x}-h_S\prn{w}}/h_E\prn{w}$ is $\prn{1/\mu}$-Lipschitz in $\norm{\cdot}_\infty$.
The pointwise maximum $M$ of these functions over $w\in\triangle_d$ is again $\prn{1/\mu}$-Lipschitz, and so is $d_E=\pos{M}$.
In particular, $d_E$ is continuous on $\R^d_+$.

Since $U$ is compact and $d_E$ is continuous, $\max_{x\in U}d_E\prn{x}$ exists, and interchanging the two maxima, together with the monotonicity of $\pos{\cdot}$, gives $\max_{x\in U}d_E\prn{x}=V_E\prn{U,S}$ by the definition of the excess in~\eqref{eq:score}.

Finally, we focus on the case $E=S$.
The convexity gives $S+\epsilon S=\prn{1+\epsilon}S$, and the gauge property above, applied to $S$ with $\lambda=1+\epsilon>0$, gives $x\in\prn{1+\epsilon}S\iff\gamma_S\prn{x}\le1+\epsilon$. Hence $d_S\prn{x}\le\epsilon\iff\pos{\gamma_S\prn{x}-1}\le\epsilon$ for all $\epsilon\ge0$, and the two values coincide.
\end{proof}

The $L^2$ and $L^\infty$ rules both rest on the following common fact, which reduces the one-sided distance $d_{E_p}$ to the usual $p$-distance to $S$.
Note that $S$ is a safe region, and its geometric properties are introduced in \cref{ass:sets}.
\begin{lemma}
\label{lem:norm-enlargement}
Under \cref{ass:sets}, let $p\in\set{2,\infty}$ and $x\in\R^d_+$. Then there is a point $q\in S$ nearest to $x$ in $\norm{\cdot}_p$ with $q\le x$ coordinatewise, and $d_{E_p}\prn{x}=\dist_p\prn{x,S}$.
\end{lemma}
\begin{proof}
Compactness of $S$ ensures the existence of a nearest point $s_*$.
Replacing $s_*$ by its coordinatewise minimum with $x$ preserves membership in $S$ by downward closedness, and cannot increase the distance because each coordinate
of $x-s_*$ either stays the same or becomes zero.
Since $\gamma_{E_p}\prn{z}=\norm{z}_p$ if $z\ge0$ and $\gamma_{E_p}\prn{z}=+\infty$ otherwise (which is basic), we have $d_{E_p}\prn{x}=\min\Set{\norm{x-q}_p}{q\in S,\ q\le x}$, which equals $\dist_p\prn{x,S}$ by the first claim.
\end{proof}

All three rules ($L^2$, $L^\infty$, and gauge) certify the optimality of a direction $w$ in the same way.
The following lemma characterizes a sufficient condition of $w$ to maximize the direction score $v_E$.
\begin{lemma}
\label{lem:certificate}
Under \cref{ass:sets}, let $E\subseteq\R_+^d$ be an error unit that is nonempty, compact, convex, downward closed, and satisfies $h_E>0$ over $\triangle_d$. Let $x_*\in U$ maximize $d_E$ over $U$, and let $w\in\triangle_d$ satisfy
\begin{equation}
\inpr{w,x_*}-h_S\prn{w}\ge d_E\prn{x_*}\,h_E\prn{w}.
\label{eq:certificate}
\end{equation}
Then we have $h_U\prn{w}=\inpr{w,x_*}$, the inequality~\eqref{eq:certificate} holds with equality, and $w$ maximizes $v_E$.
\end{lemma}
\begin{proof}
We have $h_U\prn{w}\ge\inpr{w,x_*}$ by definition of the support function.
This together with the assumption \eqref{eq:certificate} and \cref{prop:infconv} leads to
\[
v_E\prn{w}=\frac{h_U\prn{w}-h_S\prn{w}}{h_E\prn{w}}\ge\frac{\inpr{w,x_*}-h_S\prn{w}}{h_E\prn{w}}\ge d_E\prn{x_*}=V_E\prn{U,S}\ge v_E\prn{w},
\]
where the last inequality follows from the definition of the excess~\eqref{eq:score}.
Hence all inequalities are equalities, which gives $V_E(U,S) = v_E(w)$, or equivalently, $w$ maximizes $v_E$.
\end{proof}

\subsection{Proofs for \texorpdfstring{$L^2$}{L2} score}
\begin{proof}[Proof of \cref{prop:l2}]
By \cref{prop:infconv,lem:norm-enlargement}, we have the reduction of the excess to the directed Hausdorff distance: $V_{E_2}\prn{U,S}=\max_{x\in U}d_{E_2}\prn{x}=\max_{x\in U}\dist_2\prn{x,S}$.

Let this value be positive and fix a maximizer $x_*$.
Since the Euclidean projection $q_*=\proj_S^2(x_*)$ is the unique nearest point, it coincides with the nearest point below $x_*$ given by \cref{lem:norm-enlargement}, that is, $q_*\le x_*$.
Moreover, $x_*\neq q_*$ because $\norm{x_*-q_*}_2=V_{E_2}\prn{U,S}>0$.
Thus $x_*-q_*$ is a nonzero nonnegative vector, and its $\ell_1$-normalization $w_*$ in \eqref{eq:l2-direction} belongs to $\triangle_d$. The projection inequality gives $\inpr{x_*-q_*,s-q_*}\le0$ for all $s\in S$, and dividing both sides by $\norm{x_*-q_*}_1$ gives $h_S\prn{w_*}=\inpr{w_*,q_*}$.
Since $w_*$ and $x_*-q_*$ are parallel, and we have $h_{E_2}\prn{w_*}=\norm{w_*}_2$, we obtain
\begin{equation}
\inpr{w_*,x_*}-h_S\prn{w_*}=\inpr{w_*,x_*-q_*}=\norm{w_*}_2\norm{x_*-q_*}_2=h_{E_2}\prn{w_*}\,d_{E_2}\prn{x_*},
\end{equation}
so $w_*$ maximizes the direction score $v_{E_2}$ by \cref{lem:certificate}.
\end{proof}

\subsection{Proofs for \texorpdfstring{$L^\infty$}{L-infty} score}
\begin{proof}[Proof of \cref{prop:linf}]
By \cref{prop:infconv,lem:norm-enlargement}, we have the reduction of the excess to the directed Hausdorff distance: $V_{E_\infty}\prn{U,S}=\max_{x\in U}d_{E_\infty}\prn{x}=\max_{x\in U}\dist_\infty\prn{x,S}$.

We check that $q_*\coloneqq\pos{x_*-\delta\mathbf1}$ lies in $S$ and is a nearest point to $x_*$ in the $L^\infty$ sense.
By \cref{lem:norm-enlargement}, some nearest point $q\le x_*$ exists, and $q_*\le q$ holds coordinatewise because $q_{*,i}=\max\set{x_{*,i}-\delta,0}\le q_i$ follows from $q_i\ge0$ and $x_{*,i}-q_i\le\norm{x_*-q}_\infty=\delta$.
Downward closedness then gives $q_*\in S$. Since $x_{*,i}-q_{*,i}=\min\set{x_{*,i},\delta}\le\delta$ for every $i$, we have $\norm{x_*-q_*}_\infty\le\delta=\dist_\infty\prn{x_*,S}$, so $q_*$ is also a nearest point.

Next, we show the existence of a direction $w_*\in\triangle_d$ satisfying the certificate~\eqref{eq:linf-conditions}.
By the choice of $x_*$, the excess equals $\delta$; additionally, $\delta>0$ holds by assumption.
Since $h_{E_\infty}=1$ on $\triangle_d$ and $\delta>0$, \eqref{eq:infconv-dual} at $x_*$ reads $\delta=d_{E_\infty}\prn{x_*}=\max_{w\in\triangle_d}\prn{\inpr{w,x_*}-h_S\prn{w}}$, and we let $w_*\in\triangle_d$ attain this maximum.
Since $q_*\in S$, the definition of the support function and H\"older's inequality give
\begin{equation}
\delta=\inpr{w_*,x_*}-h_S\prn{w_*}
\le\inpr{w_*,x_*-q_*}
\le\norm{w_*}_1\norm{x_*-q_*}_\infty=\delta.
\label{eq:linf-upper}
\end{equation}
Both inequalities are therefore equalities, which gives the two conditions in the certificate~\eqref{eq:linf-conditions}.

Conversely, every $w_*\in\triangle_d$ satisfying \eqref{eq:linf-conditions} has $\inpr{w_*,x_*}-h_S\prn{w_*}=\inpr{w_*,x_*-q_*}=\delta=h_{E_\infty}\prn{w_*}\,d_{E_\infty}\prn{x_*}$,
where the last identity follows from $h_{E_\infty}=1$ and $\delta=d_{E_\infty}(x_*)$.
Therefore it maximizes $v_{E_\infty}$ by \cref{lem:certificate}.
\end{proof}

Unlike the $L^2$ case in \cref{prop:l2}, we do not have $w_*$ in closed form, which is defined implicitly through the maximization problem in \eqref{eq:infconv-dual}.
Subsequently, we consider a smooth safe region to get a concise form of $w_*$ in \cref{prop:linf-smooth}.
To this end, we need the following lemma to reduce the certificate~\eqref{eq:linf-conditions} to a condition on the normal cone of $S$ at $q_*$.
Throughout, $I\coloneqq\Set{i}{x_{*,i}-q_{*,i}=\delta}$ and $\mathbf1_I\in\set{0,1}^d$ is its indicator vector, as in \cref{sec:three}.

\begin{lemma}
\label{lem:linf-saturated}
In the setting of \cref{prop:linf}, the set $I$ is nonempty and $q_{*,i}=0$ for every $i\notin I$.
Moreover, $w\in\triangle_d$ satisfies \eqref{eq:linf-conditions} if and only if
$h_S\prn{w}=\inpr{w,q_*}$ and $w_i=0$ for every $i\notin I$.
\end{lemma}
\begin{proof}
By the definition of $q_*$, we have $x_{*,i}-q_{*,i}=\min\set{x_{*,i},\delta}\le\delta$ for every $i$, with equality exactly on $I$.
Since $\norm{x_*-q_*}_\infty=\delta$, the set $I$ is nonempty.
For $i\notin I$, we have $x_{*,i}<\delta$ and hence $q_{*,i}=0$.
For $w\in\triangle_d$, we have
$\inpr{w,x_*-q_*}=\sum_iw_i\prn{x_{*,i}-q_{*,i}}\le\delta\sum_iw_i=\delta$,
with equality if and only if $w_i=0$ for every $i\notin I$.
\end{proof}

Since $q_*\in S$, the condition $h_S\prn{w}=\inpr{w,q_*}$ means that $w$ belongs to the normal cone 
$$N_S\prn{q_*}\coloneqq\Set{w\in\R^d}{\inpr{w,s-q_*}\le0\ \forall s\in S}.$$
By \cref{lem:linf-saturated}, the directions $w_*$ in \cref{prop:linf} are exactly the elements of $N_S\prn{q_*}\cap\triangle_d$ supported on $I$:
$w_*\in N_S(q_*)$ corresponds to the first certificate in \eqref{eq:linf-conditions}, and $w_{*,i}=0$ for every $i\not\in I$ corresponds to the second certificate in \eqref{eq:linf-conditions}.
With this lemma, we have a closed form of $w_*$ for smooth $S$, as shown below.
\begin{proof}[Proof of \cref{prop:linf-smooth}]
Based on \cref{lem:linf-saturated}, we consider the characterization of $w_*$ by the normal cone $N_S$.
By definition, $w\in N_S\prn{q_*}$ if and only if $q_*$ solves
\begin{equation}
\max_{s\in\R^d}\ \inpr{w,s}
\quad\text{subject to}\quad
f\prn{s}\le\theta,\quad
\inpr{\mathbf1,s}\le1,\quad
s_i\ge0\ \ \prn{i\in[d]}.
\label{eq:linf-normal-problem}
\end{equation}
All constraint functions in \eqref{eq:linf-normal-problem} are convex and differentiable, and the point $\eta\mathbf1$ satisfies all of them strictly for small $\eta>0$ by $f\prn{0}<\theta$ and continuity; that is, Slater's condition holds for \eqref{eq:linf-normal-problem}.
Hence the KKT conditions are necessary and sufficient for the optimality of $q_*$, which gives
\[
N_S\prn{q_*}=\Set{\lambda\nabla f\prn{q_*}+\nu\mathbf1-\textstyle\sum_{i\in Z}\mu_ie_i}{\lambda,\nu\ge0,\ \mu\ge0},
\]
where $e_i\in\R^d$ is the $i$-th standard basis vector and $Z\coloneqq\Set{i}{q_{*,i}=0}$.
By complementary slackness,
$\lambda=0$ if $f\prn{q_*}<\theta$,
and $\nu=0$ if $\inpr{\mathbf1,q_*}<1$.
Since $f$ is nondecreasing in each coordinate, $\nabla f\prn{q_*}\ge0$.

We first show that the simplex constraint is inactive at $q_*$.
By definition, $q_*\le x_*$, and $q_{*,i}=x_{*,i}-\delta<x_{*,i}$ for $i\in I$, where $I$ is nonempty by \cref{lem:linf-saturated}.
Since $x_*\in U\subseteq\triangle_d^{\le}$, we obtain $\inpr{\mathbf1,q_*}<\inpr{\mathbf1,x_*}\le1$, and thus $\nu=0$.

Next, we show $\nabla f\prn{q_*}\odot\mathbf1_I\neq0$.
Take $w$ from \cref{prop:linf} and write $w=\lambda\nabla f\prn{q_*}-\sum_{i\in Z}\mu_ie_i$ with $\lambda,\mu\ge0$.
By \cref{lem:linf-saturated}, $w$ is supported on $I$, so
\[
1=\sum_{i\in I}w_i\le\lambda\sum_{i\in I}\nabla_if\prn{q_*}.
\]
Hence $\nabla f\prn{q_*}\odot\mathbf1_I\neq0$.
In particular, $\lambda>0$, so $f\prn{q_*}=\theta$ and $\nabla f\prn{q_*}\in N_S\prn{q_*}$.

Since $\set{1,\dots,d}\setminus I\subseteq Z$ by \cref{lem:linf-saturated}, we have
\[
\nabla f\prn{q_*}\odot\mathbf1_I=\nabla f\prn{q_*}-\sum_{i\notin I}\nabla_if\prn{q_*}e_i\in N_S\prn{q_*},
\]
with $\nabla_if\prn{q_*}\ge0$.
Hence $w_*$ belongs to $N_S\prn{q_*}\cap\triangle_d$ and is supported on $I$, and \cref{lem:linf-saturated} gives \eqref{eq:linf-conditions}.
\end{proof}

\subsection{Proofs for gauge score}
As in the $L^\infty$ case (\cref{prop:linf}), we first show the existence of the direction score maximizer.
\begin{proof}[Proof of \cref{prop:gauge}]
Let $g\prn{x}\coloneqq\max_{w\in\triangle_d}\inpr{w,x}/h_S\prn{w}$ for $x\in\R^d_+$.
Since $h_E=h_S$ for $E=S$, \eqref{eq:infconv-dual} and \cref{prop:infconv} give $\pos{\gamma_S\prn{x}-1}=d_S\prn{x}=\pos{g\prn{x}-1}$, so $\gamma_S\prn{x}=g\prn{x}$ whenever either exceeds $1$.
Both functions are positively homogeneous and positive at every $x\ne0$, because $S$ is bounded and $h_S>0$ on $\triangle_d$; hence scaling $x$ gives $\gamma_S=g$ on $\R^d_+$.
As in the proof of \cref{prop:infconv}, $g$ is Lipschitz continuous, so $\gamma_S$ attains its maximum over the compact set $U$, and $V_S\prn{U,S}=\max_{x\in U}d_S\prn{x}=\pos{\max_{x\in U}\gamma_S\prn{x}-1}$ by \cref{prop:infconv}. This proves \eqref{eq:gauge-excess}.

If the excess is positive, then $\lambda_*=\gamma_S(x_*)>1$ and $d_S\prn{x_*}=\lambda_*-1$. By positive homogeneity, $\gamma_S\prn{q_*}=\gamma_S\prn{x_*}/\lambda_*=1$, so $d_S\prn{q_*}=0$ and $q_*\in S$ by \cref{prop:infconv}. For $t>1$, $\gamma_S\prn{tq_*}=t>1$ gives $tq_*\notin S$; letting $t\downarrow1$ shows that $q_*$ lies on the boundary of $S$.

Every $w\in\triangle_d$ satisfying $h_S\prn{w}=\inpr{w,q_*}$ obeys, by $x_*=\lambda_*q_*$,
\begin{equation}
\inpr{w,x_*}-h_S\prn{w}=\prn{\lambda_*-1}h_S\prn{w}=d_S\prn{x_*}\,h_S\prn{w},
\end{equation}
so \cref{lem:certificate} with $E=S$ gives $\inpr{w,x_*}=h_U\prn{w}$ and the maximality of $w$.
\end{proof}

Next, we consider a smooth safe region $S$ to get a closed form of the direction score maximizer $w_*$.
\begin{proof}[Proof of \cref{prop:gauge-smooth}]
Continuity and $f\prn{0}<\theta$ give $\rho\triangle_d^{\le}\subseteq S$ for some $\rho>0$, so \eqref{eq:positivity} holds. By \cref{prop:gauge}, $q_*\in S$, and $\sum_i q_{*,i}=\lambda_*^{-1}\sum_i x_{*,i}<1$. If $f\prn{q_*}<\theta$, continuity gives a number $\eta>1$ sufficiently close to one that $\eta q_*\in\triangle_d^{\le}$ and $f\prn{\eta q_*}<\theta$. Then $\eta q_*\in S$ and $x_*\in\prn{\lambda_*/\eta}S$, contradicting the definition of $\lambda_*$. Thus $f\prn{q_*}=\theta$.

Monotonicity of $f$ gives $\nabla f\prn{q_*}\ge0$. For each $x\in S$ and $t\in\brc{0,1}$, convexity of $S$ implies $f\prn{q_*+t\prn{x-q_*}}\le\theta=f\prn{q_*}$. Taking the right derivative at $t=0$ yields $\inpr{\nabla f\prn{q_*},x-q_*}\le0$. If the gradient is nonzero, its normalization therefore belongs to $\triangle_d$ and supports $S$ at $q_*$. The conclusion follows from \cref{prop:gauge}.
\end{proof}

\end{revision}
\begin{revision}
\section{Computations for the Example}
\label{app:examples}

For the sets in \cref{ex:boxes} and $w\in\triangle_2$, the support functions are
\begin{equation}
h_U\prn{w}=\frac{3w_1+w_2}{4},\qquad
h_S\prn{w}=\frac{\max\set{w_1,w_2}}{4}.
\label{eq:example-supports}
\end{equation}
The score in \eqref{eq:score} divides their difference by
$\norm{w}_2$, $1$, or $h_S\prn{w}$ for the three units in
\eqref{eq:three-units}, respectively.
Since $\max\set{w_1,w_2}\ge w_1$, \eqref{eq:example-supports} gives
$h_U\prn{w}-h_S\prn{w}\le\prn{2w_1+w_2}/4$.
The Cauchy--Schwarz inequality therefore yields
$v_{E_2}\prn{w}\le\sqrt5/4$, with equality at $w=\prn{2/3,1/3}$.
Since $w_1+w_2=1$, the same support-gap bound gives
$v_{E_\infty}\prn{w}\le1/2$, with equality at $w=\prn{1,0}$.
Finally, \eqref{eq:example-supports} gives $h_U\prn{w}\le4h_S\prn{w}$,
so $v_S\prn{w}\le3$, with equality at $w=\prn{1/2,1/2}$.

The point $q_* = \prn{1/4,0}$ is the Euclidean projection of
$x_* = \prn{3/4,1/4}$ because
$\inpr{x_*-q_*,s-q_*} = \prn{2s_1+s_2}/4-1/8\le0$
for every $s\in S$.
Its distance from $x_*$ is $\sqrt5/4$, matching the maximal distance
in \cref{prop:l2}.
The $L^\infty$ distance from $x_*$ to $S$ is $1/2$:
every $s\in S$ satisfies $s_1\le1/4$, and $q_*$ attains this bound.
This matches the maximal distance in \cref{prop:linf}, and
$q_* = \pos{x_* - \prn{1/2}\mathbf1}$.
For $w = \prn{1/2,1/2}$, we have
$h_S\prn{w}=\inpr{w,q_*}=1/8$, so $w$ supports $S$ at $q_*$,
but $v_{E_\infty}\prn{w}=1/2-1/8=3/8<1/2$.
\end{revision}

\section{Omitted Details of Experiments}
\label{app:experimental-details}
In~\cref{tab:all-natural-query-counts}, we report the number of
oracle calls made for each benchmark--method pair.  For the safe Gridworld instances, the $L^2$ and $L^\infty$ methods
return \textsc{unknown} at dimension $d=10$, before reaching the refinement-iteration limit.
These early terminations reflect numerical limitations of our
current implementation.
In the former case, a safe-region support computation failed
the numerical optimality check.
In the latter, outer-polytope vertex enumeration encountered
a numerical failure, and the combinatorial fallback was not
attempted because the problem exceeded its configured size limit.

\begin{table*}[t]
\centering
\caption{Oracle calls on all 24 benchmark cases. The dimension $d$ is the number of target states in MDPs. A dagger (${}^\dagger$) marks an \UNKNOWN{} result.}
\label{tab:all-natural-query-counts}
\small
\setlength{\tabcolsep}{6pt}
\begin{center}
\begin{tabular}{@{}llrrrrrr@{}}
\toprule
Benchmark & Expected & $d$ & $L^2$ & $L^\infty$
& Gauge & LK-$\epsilon$ & LK-$\gamma$ \\
\midrule
Gridworld & \SAFE{} & 4
& 8 & 11 & 3 & 14 & 6 \\
Gridworld & \SAFE{} & 6
& 16 & 16 & 3 & 49 & 9 \\
Gridworld & \SAFE{} & 8
& 38 & $50^\dagger$ & 6 & $50^\dagger$ & 16 \\
Gridworld & \SAFE{} & 10
& $23^\dagger$ & $42^\dagger$ & 4 & $50^\dagger$ & 26 \\
\addlinespace
Robot mission & \SAFE{} & 4
& 5 & 5 & 4 & 6 & 6 \\
Robot mission & \SAFE{} & 6
& 7 & 7 & 6 & 7 & 7 \\
Robot mission & \SAFE{} & 8
& 9 & 9 & 8 & $50^\dagger$ & 37 \\
Robot mission & \SAFE{} & 10
& 11 & 11 & 11 & $50^\dagger$ & $50^\dagger$ \\
\addlinespace
Network routing & \SAFE{} & 4
& 9 & 13 & 5 & 15 & 15 \\
Network routing & \SAFE{} & 6
& 25 & 23 & 7 & $50^\dagger$ & 40 \\
Network routing & \SAFE{} & 8
& 46 & 47 & 9 & $50^\dagger$ & $50^\dagger$ \\
Network routing & \SAFE{} & 10
& $50^\dagger$ & $50^\dagger$ & 12
& $50^\dagger$ & $50^\dagger$ \\
\midrule
Gridworld & \UNSAFE{} & 4
& 1 & 1 & 1 & 1 & 1 \\
Gridworld & \UNSAFE{} & 6
& 1 & 1 & 1 & 1 & 1 \\
Gridworld & \UNSAFE{} & 8
& 1 & 1 & 1 & 1 & 1 \\
Gridworld & \UNSAFE{} & 10
& 1 & 1 & 1 & 1 & 1 \\
\addlinespace
Robot mission & \UNSAFE{} & 4
& 5 & 5 & 2 & 16 & 16 \\
Robot mission & \UNSAFE{} & 6
& 7 & 7 & 1 & $50^\dagger$ & $50^\dagger$ \\
Robot mission & \UNSAFE{} & 8
& 9 & 9 & 4 & $50^\dagger$ & $50^\dagger$ \\
Robot mission & \UNSAFE{} & 10
& 11 & 11 & 1 & $50^\dagger$ & $50^\dagger$ \\
\addlinespace
Network routing & \UNSAFE{} & 4
& 1 & 1 & 1 & 1 & 1 \\
Network routing & \UNSAFE{} & 6
& 3 & 3 & 3 & 3 & 3 \\
Network routing & \UNSAFE{} & 8
& 9 & 9 & 1 & 34 & 15 \\
Network routing & \UNSAFE{} & 10
& 9 & 9 & 7 & 1 & 1 \\
\bottomrule
\end{tabular}

\par\smallskip
\end{center}
\end{table*}

\section{Relation to Quality Indicator}
\label{sec:indicatorEnc}
We relate the quality indicators introduced
by~\citet[Definitions~2.1 and~2.2, and Section~5.2]{lammel2023convergence}
to our excess $V_E\prn{U,S}$ in \eqref{eq:score}.
Their quality indicators measure the difference of outer and inner approximations of a Pareto front, while our excess is used to measure the difference of an outer approximation and a safe region.
Let $I\subseteq O$ be inner and outer convex polytopes satisfying
\cref{ass:sets} with $S=I$ and $U=O$.

In this convention, the epsilon indicator is $\delta_\epsilon(I,O)
\coloneqq
\min\{\epsilon\ge0:O\subseteq I+\epsilon[0, 1]^d\}$. Thus, by~\cref{thm:containment},
we have 
$\delta_\epsilon(I,O)=V_{[0,1]^d}(O,I)$. For the polyhedral gauge, additionally assume that $I$
satisfies \eqref{eq:positivity}, and define
$\gamma_I(z)\coloneqq
\min\{\lambda\ge0:z\in\lambda I\}$.
The associated approximation quality indicator~\citep[Section~5.2]{lammel2023convergence} is
\[
\delta_\gamma(I,O)
\coloneqq
\left[\max_{z\in O}\gamma_I(z)-1\right]_+.
\]
Since $I+\epsilon I=(1+\epsilon)I$, we obtain
\[
\delta_\gamma(I,O)
=
\min\{\epsilon\ge0:O\subseteq(1+\epsilon)I\}
=
V_I(O,I)
\]
by~\cref{thm:containment}.

\end{document}